\documentclass[11pt]{amsart}

\usepackage[all]{xy}
\usepackage{color, pstricks}
\usepackage{mathrsfs}
\usepackage{amsmath,amsthm, amssymb, amsgen,amsxtra,amsfonts, amsbsy} 
\usepackage[all]{xy}
\usepackage{verbatim}
\usepackage{url}
\usepackage{hyperref}

\begin{document}
%
%
%
\theoremstyle{definition}
\newtheorem{Definition}{Definition}[section]
\newtheorem*{Definitionx}{Definition}
\newtheorem{Convention}[Definition]{Convention}
\newtheorem{Construction}{Construction}[section]
\newtheorem{Example}[Definition]{Example}
\newtheorem{Examples}[Definition]{Examples}
\newtheorem{Remark}[Definition]{Remark}
\newtheorem{Setup}[Definition]{Setup}
\newtheorem*{Remarkx}{Remark}
\newtheorem{Remarks}[Definition]{Remarks}
\newtheorem{Caution}[Definition]{Caution}
\newtheorem{Conjecture}[Definition]{Conjecture}
\newtheorem*{Conjecturex}{Conjecture}
\newtheorem{Question}[Definition]{Question}
\newtheorem{Questions}[Definition]{Questions}
\newtheorem*{Acknowledgements}{Acknowledgements}
\newtheorem*{Organization}{Organization}
\newtheorem*{Disclaimer}{Disclaimer}
\theoremstyle{plain}
\newtheorem{Theorem}[Definition]{Theorem}
\newtheorem*{Theoremx}{Theorem}
\newtheorem{Theoremy}{Theorem}
\newtheorem{Proposition}[Definition]{Proposition}
\newtheorem*{Propositionx}{Proposition}
\newtheorem{Lemma}[Definition]{Lemma}
\newtheorem{Corollary}[Definition]{Corollary}
\newtheorem*{Corollaryx}{Corollary}
\newtheorem{Fact}[Definition]{Fact}
\newtheorem{Facts}[Definition]{Facts}
\newtheoremstyle{voiditstyle}{3pt}{3pt}{\itshape}{\parindent}%
{\bfseries}{.}{ }{\thmnote{#3}}%
\theoremstyle{voiditstyle}
\newtheorem*{VoidItalic}{}
\newtheoremstyle{voidromstyle}{3pt}{3pt}{\rm}{\parindent}%
{\bfseries}{.}{ }{\thmnote{#3}}%
\theoremstyle{voidromstyle}
\newtheorem*{VoidRoman}{}

%
\newcommand{\prf}{\par\noindent{\sc Proof.}\quad}
\newcommand{\blowup}{\rule[-3mm]{0mm}{0mm}}
\newcommand{\cal}{\mathcal}
\newcommand{\Aff}{{\mathds{A}}}
\newcommand{\BB}{{\mathds{B}}}
\newcommand{\CC}{{\mathds{C}}}
\newcommand{\EE}{{\mathds{E}}}
\newcommand{\FF}{{\mathds{F}}}
\newcommand{\GG}{{\mathds{G}}}
\newcommand{\HH}{{\mathds{H}}}
\newcommand{\NN}{{\mathds{N}}}
\newcommand{\ZZ}{{\mathds{Z}}}
\newcommand{\PP}{{\mathds{P}}}
\newcommand{\QQ}{{\mathds{Q}}}
\newcommand{\RR}{{\mathds{R}}}
\newcommand{\Sph}{{\mathds{S}}}
\newcommand{\TT}{{\mathds{T}}}
\newcommand{\Liea}{{\mathfrak a}}
\newcommand{\Lieb}{{\mathfrak b}}
\newcommand{\Lieg}{{\mathfrak g}}
\newcommand{\Liem}{{\mathfrak m}}
\newcommand{\ideala}{{\mathfrak a}}
\newcommand{\idealb}{{\mathfrak b}}
\newcommand{\idealg}{{\mathfrak g}}
\newcommand{\idealj}{{\mathfrak j}}
\newcommand{\idealm}{{\mathfrak m}}
\newcommand{\idealn}{{\mathfrak n}}
\newcommand{\idealp}{{\mathfrak p}}
\newcommand{\idealq}{{\mathfrak q}}
\newcommand{\idealI}{{\cal I}}
\newcommand{\lin}{\sim}
\newcommand{\num}{\equiv}
\newcommand{\dual}{\ast}
\newcommand{\iso}{\cong}
\newcommand{\homeo}{\approx}
\newcommand{\mathds}[1]{{\mathbb #1}}
\newcommand{\mm}{{\mathfrak m}}
\newcommand{\pp}{{\mathfrak p}}
\newcommand{\qq}{{\mathfrak q}}
\newcommand{\rr}{{\mathfrak r}}
\newcommand{\pP}{{\mathfrak P}}
\newcommand{\qQ}{{\mathfrak Q}}
\newcommand{\rR}{{\mathfrak R}}
%
%
\newcommand{\OO}{{\cal O}}
\newcommand{\calA}{{\cal A}}
\newcommand{\calD}{{\cal D}}
\newcommand{\calM}{{\cal M}}
\newcommand{\calO}{{\cal O}}
\newcommand{\calP}{{\cal P}}
\newcommand{\calT}{{\cal T}}
\newcommand{\calU}{{\cal U}}
\newcommand{\numero}{{n$^{\rm o}\:$}}
\newcommand{\mf}[1]{\mathfrak{#1}}
\newcommand{\mc}[1]{\mathcal{#1}}
\newcommand{\into}{{\hookrightarrow}}
\newcommand{\onto}{{\twoheadrightarrow}}
\newcommand{\Spec}{{\rm Spec}\:}
\newcommand{\BigSpec}{{\rm\bf Spec}\:}
\newcommand{\Spf}{{\rm Spf}\:}
\newcommand{\Proj}{{\rm Proj}\:}
\newcommand{\Pic}{{\rm Pic }}
\newcommand{\Picloc}{{\rm Picloc }}
\newcommand{\Br}{{\rm Br}}
\newcommand{\NS}{{\rm NS}}
\newcommand{\id}{{\rm id}}
\newcommand{\Sym}{{\mathfrak S}}
\newcommand{\Aut}{{\rm Aut}}
\newcommand{\Autp}{{\rm Aut}^p}
\newcommand{\End}{{\rm End}}
\newcommand{\Hom}{{\rm Hom}}
\newcommand{\Ext}{{\rm Ext}}
\newcommand{\ord}{{\rm ord}}
\newcommand{\coker}{{\rm coker}\,}
\newcommand{\divisor}{{\rm div}}
\newcommand{\Def}{{\rm Def}}
\newcommand{\et}{{\rm \acute{e}t}}
\newcommand{\loc}{{\rm loc}}
\newcommand{\ab}{{\rm ab}}
\newcommand{\piet}{{\pi_1^{\rm \acute{e}t}}}
\newcommand{\pietloc}{{\pi_{\rm loc}^{\rm \acute{e}t}}}
\newcommand{\piN}{{\pi^{\rm N}_1}}
\newcommand{\piNloc}{{\pi_{\rm loc}^{\rm N}}}
\newcommand{\Het}[1]{{H_{\rm \acute{e}t}^{{#1}}}}
\newcommand{\Hfl}[1]{{H_{\rm fl}^{{#1}}}}
\newcommand{\Hcris}[1]{{H_{\rm cris}^{{#1}}}}
\newcommand{\HdR}[1]{{H_{\rm dR}^{{#1}}}}
\newcommand{\hdR}[1]{{h_{\rm dR}^{{#1}}}}
\newcommand{\Torsloc}{{\rm Tors}_{\rm loc}}
\newcommand{\defin}[1]{{\bf #1}}
\newcommand{\oX}{\cal{X}}
\newcommand{\oA}{\cal{A}}
\newcommand{\oY}{\cal{Y}}
\newcommand{\calC}{{\cal{C}}}
\newcommand{\calL}{{\cal{L}}}
\newcommand{\bmu}{\boldsymbol{\mu}}
\newcommand{\balpha}{\boldsymbol{\alpha}}
\newcommand{\bL}{{\mathbf{L}}}
\newcommand{\bM}{{\mathbf{M}}}
\newcommand{\bW}{{\mathbf{W}}}
\newcommand{\bD}{{\mathbf{D}}}
\newcommand{\bT}{{\mathbf{T}}}
\newcommand{\bO}{{\mathbf{O}}}
\newcommand{\bI}{{\mathbf{I}}}
\newcommand{\BD}{{\mathbf{BD}}}
\newcommand{\BT}{{\mathbf{BT}}}
\newcommand{\BI}{{\mathbf{BI}}}
\newcommand{\BO}{{\mathbf{BO}}}
\newcommand{\C}{{\mathbf{C}}}
\newcommand{\Dic}{{\mathbf{Dic}}}
\newcommand{\SL}{{\mathbf{SL}}}
\newcommand{\MC}{{\mathbf{MC}}}
\newcommand{\GL}{{\mathbf{GL}}}
\newcommand{\Tors}{{\mathbf{Tors}}}

\newcommand{\TY}[1]{\textcolor{blue}{(TY: #1)}}

\newcommand{\CL}[1]{\textcolor{red}{(CL: #1)}}

\makeatletter
\@namedef{subjclassname@2020}{\textup{2020} Mathematics Subject Classification}
\makeatother

\title[NCRs of LRQ singularities]{Non-commutative resolutions of linearly reductive quotient singularities}

\author{Christian Liedtke}
\address{TU M\"unchen, Zentrum Mathematik - M11, Boltzmannstr. 3, 85748 Garching bei M\"unchen, Germany}
\curraddr{Mathematical Institute, University of Oxford, Radcliffe Observatory, Andrew Wiles Building, Woodstock Rd, Oxford OX2 6GG, United Kingdom}
\email{christian.liedtke@tum.de}
\email{christian.liedtke@maths.ox.ac.uk}

\author{Takehiko Yasuda}
\address{Department of Mathematics, Graduate School of Science, Osaka University
Toyonaka, Osaka 560-0043, Japan}
\email{yasuda.takehiko.sci@osaka-u.ac.jp}


\begin{abstract}
We prove existence of non-commutative crepant resolutions (in the sense of van den Bergh) of quotient singularities
by finite and linearly reductive group schemes in positive characteristic. 
In dimension two, we relate these to resolutions of singularities provided by $G$-Hilbert schemes and F-blowups.
As an application, we establish and recover results concerning resolutions for toric
singularities, as well as canonical, log terminal, and F-regular singularities in dimension 2.
\end{abstract}
\setcounter{tocdepth}{1}
\maketitle
\tableofcontents
\section{Introduction}

\subsection{Non-commutative (crepant) resolutions}
Van den Bergh \cite{vdBergh} introduced the notion of a 
\emph{non-commutative (crepant) resolution} of a singularity and we refer
to \cite{Leuschke, vdBergh ICM, Wemyss} for motivation, results, and survey.
Following \cite{vdBergh, BO, IyamaWemyss}, we make the following definition.

\begin{Definition}\label{def-NCCR-intro}
 Let $R$ be a Noetherian, local, complete, and Cohen-Macaulay commutative ring and let $M$ be a nonzero reflexive $R$-module. The endomorphism ring $\End_R(M)$ is called 
 \begin{enumerate}
 \item a \emph{non-commutative resolution} (NCR for short) if $\End_R(M)$ has finite global dimension, and
 \item a \emph{non-commutative crepant resolution} (NCCR for short) if $\End_R(M)$ has  global dimension equal to $\dim R$ and $\End_R(M)$ is 
 a Cohen-Macaulay $R$-module.
 \end{enumerate}
\end{Definition}

By now, there is a lot of work dedicated to NCRs and NCCRs, such as 
for the following classes of singularities:
\begin{enumerate}
\item  quotient singularities by finite groups in characteristic zero \cite{vdBergh}
or by finite groups of order prime to the characteristic \cite{TY},
\item quotient singularities by not necessarily finite reductive group schemes in characteristic zero \cite{Spenko},
\item hypersurface singularities \cite{Dao}, and
\item toric singularities \cite{Faber}, \cite{Spenko2}, \cite{Spenko3}, \cite{Spenko4}.
\end{enumerate}

\subsection{Linearly reductive quotient singularities}\label{sec: lin red quot sings}
In this article, we study NCCRs in the following situation.

\begin{Setup}\label{Setup}
Let $k$ be an algebraically closed field of characteristic $p>0$.
Let $S:=k[[x_1,...,x_n]]$ be a formal power series ring over $k$.
Let $G$ be a finite group scheme over $k$ that acts on $\Spec S$, 
such that the action is free in codimension one.
We let 
$$
  R\,:=\,S^G \,\subseteq\, S
$$
be the invariant subring and set $X:=\Spec R$.
\end{Setup}

\begin{Definition}
 A finite group scheme $G$ over $k$ is called \emph{linearly reductive} if every
 $k$-linear and finite-dimensional representation of $G$ is semi-simple.
\end{Definition}

\begin{Definition}
 A scheme $X=\Spec R$ as in Setup \ref{Setup}
 is called a \emph{quotient singularity} by the group scheme $G$.
 If $G$ is linearly reductive, then $X$ is called a \emph{linearly reductive quotient singularity}
 (LRQ singularity for short).
\end{Definition}

For background and details on quotient singularities by finite group schemes and especially
linearly reductive ones, we refer to \cite{AOV,LMM, LMM2, Liedtke, LieSa}.

\subsection{NCCRs for LRQ singularities}
In Section \ref{sec: skew group scheme rings}, we first establish the following result,
which extends a classical proposition of Auslander \cite{Auslander}.
Our extension was essentially already obtained by Faber,
Ingalls, Okawa, and Satriano \cite[Proposition 2.26]{FIOS}, 
although they work in a slightly different setup.
Moreover, a similar result was proven in \cite{BHZ,CKWZ}, again 
in somewhat different setups.

\begin{Theorem}
\label{thm: intro1}
 There exists a \emph{skew group scheme ring}
 $$
   S*G \,=\, S\#H,
 $$
 where $H$ is the dual Hopf algebra of $H^0(G,\calO_G)$.
 Moreover, there exists a natural isomorphism
 $$
   S*G \,\to\, \End_R(S)
 $$
 of usually non-commutative $R$-algebras.
 \end{Theorem}

Next, we establish the following equivalences, which generalise results
of Broer and the second-named author \cite{Broer, Yasuda}.

\begin{Theorem}\label{thm: intro-1.5}
 The following are equivalent:
 \begin{enumerate}
 \item $R$ is a pure subring of $S$.
 \item $R$ is strongly F-regular.
 \item $G$ is linearly reductive.
 \item $S*G$ has finite global dimension.
 \end{enumerate}
 Moreover, if $S*G$ has finite global dimension, then
 $\mathrm{gl.dim}(S*G)=n$.
\end{Theorem}

As an application, we obtain that LRQ singularities admit
NCCRs.

 \begin{Theorem}
  \label{thm: intro2}
   If $G$ is linearly reductive, then  
  \begin{enumerate}
 \item $\End_R(S)$ is an NCCR of $R$.
 \item If $e$ is sufficiently large, then $\End_R(R^{1/p^e})$ 
 is Morita equivalent to $\End_R(S)$ and thus, also an NCCR of $R$.
 \end{enumerate}
\end{Theorem}

This result was already known in characteristic zero, as well as for finite groups of order prime to 
the characteristic of the ground field:
in these cases, the first assertion follows, for example, from combining results of Auslander \cite{Auslander} 
and Yi \cite{Yi}, and the second assertion was established by Toda and the second-named author \cite{TY}. 

\subsection{Auslander's results and dimension two}
There are some classical results of Auslander \cite{AuslanderRS}
that carry over to the situation of this article.
To state them, we introduce the following categories:
$$
\begin{array}{lcl}
\mathrm{Rep}_k(G) & : & \mbox{finite-dimensional $k$-linear $G$-representations}\\
\calP & \qquad : \qquad& \mbox{finite and projective $S\ast G$-modules}\\
\calL & \qquad : \qquad &\mbox{finite and reflexive $R$-modules}\\
\mathrm{add}_R(S) & \qquad:\qquad& \mbox{summands of finite sums of $S$}
\end{array}
$$
In Section \ref{sec: Auslander}, we will show that these categories are related as follows.

\begin{Theorem}
\label{thm: intro3}
 Assume that $G$ is linearly reductive. 
 \begin{enumerate}
    \item The functors
     $$
     \begin{array}{ccccc}
        \mathrm{Rep}_k(G)&\to&\calP &\to& \mathrm{add}_R(S)\\
        W&\mapsto&S\otimes_kW\\
        &&P&\mapsto&P^G
      \end{array}
      $$
     induce equivalences of categories.
     Simple representations correspond to indecomposable modules
     under these equivalences.
\item If $n=2$, then the inclusion
$$
  \mathrm{add}_R(S) \,\subseteq\,\calL
$$
is an equivalence of categories.
 \end{enumerate}
 \end{Theorem}

This can be viewed as a \emph{non-commutative McKay correspondence}.

\subsection{F-blowups and G-Hilbert schemes}
For a variety $X$ over a perfect field $k$ of positive characteristic $p>0$ and an 
integer $e\geq1$, the second-named author introduced in \cite{YasudaUF}
the \emph{e.th F-blowup}
$$
   \mathrm{FB}_e(X) \,\to\,X,
$$
which is a proper and birational morphism and which is an isomorphism over the smooth locus of $X$.
In Section \ref{sec: Fblowups}, we show the following result.

\begin{Theorem}
 \label{thm: intro4}
 If $G$ is linearly reductive, then
 there exist natural, proper, and birational morphisms
 $$
   \mathrm{Hilb}^G(\Spec S) \,\stackrel{\psi_e}{\longrightarrow}\, \mathrm{FB}_e(\Spec R) \,\to\, \Spec R.
 $$
 Moreover, $\psi_e$ is an isomorphism if $e$ is sufficiently large.
 \begin{enumerate}
 \item If $n=2$, then $\mathrm{Hilb}^G(\Spec S)\to\Spec R$ is the minimal resolution of singularities.
 \item If $n=3$ and $R$ is Gorenstein, then 
 $\mathrm{Hilb}^G(\Spec S)\to\Spec R$ is a crepant resolution of singularities.
 \end{enumerate}
\end{Theorem}

This extends results of Toda and the second-named author \cite{TY,YasudaUF} from 
the case where $G$ is a finite group of order prime to $p$, the characteristic of the ground field.
Assertion (1) was already shown by the first-named author \cite{Liedtke} and extends
results of Ishii, Ito, and Nakamura \cite{Ishii, IshiiN, IN} from the case where $G$ is a finite group
of order prime to $p$. 
Assertion (2) in characteristic zero was established by Bridgeland, King, and Reid,
as well as Nakamura \cite{BKR, Nakamura}.

\subsection{Examples}
In Section \ref{sec: applications}, we apply these results to some classes of singularities:
\begin{enumerate}
\item We establish the existence of NCCRs for normal and $\QQ$-factorial 
toric singularities.
This recovers some results of Faber, Mullen, and Smith \cite{Faber} and extends some results of 
\v{S}penko and van den Bergh \cite{Spenko} to
positive characteristic.
\item We establish the existence of NCCRs
for F-regular surface singularities via their description as LRQ singularities.
This includes all canonical and log terminal singularities in dimension 2 and 
characteristic $p\geq7$.
We then recover results of Hara \cite{Hara} and the first-named author \cite{Liedtke} 
that $G$-Hilbert schemes and sufficiently high F-blowups yield the minimal
resolution of such singularities.
\end{enumerate}

Although some of these results were previously known, it is interesting that our approach
gives a natural and uniform approach via the description of these singularities as
quotients by finite and linearly reductive group schemes.

After finishing the first draft of this paper, we noticed that Hashimoto and Kobayashi 
\cite[Lemma 3.9]{HashimotoKobayashi} independently obtained a result similar to our 
Theorems \ref{thm: intro-1.5} and \ref{thm: equivalences}. 

\begin{VoidRoman}[Acknowledgements]
 We thank Mitsuyasu Hashimoto, Gregor Kemper, Gebhard Martin,
 Shinnosuke Okawa, and Michael Wemyss for discussions and comments.
 The first-named author thanks the Mathematical Institute of the University of Oxford for kind hospitality. 
 The second-named author was supported by JSPS KAKENHI Grant Numbers JP18H01112, JP21H04994,
 and JP23H01070.
\end{VoidRoman}

\section{Skew group scheme rings}
\label{sec: skew group scheme rings}

For $G$ a finite group scheme over a field $k$ of characteristic $p\geq0$
that acts on a $k$-algebra $S$, we construct
in this section the skew group scheme ring $S\ast G$.
If $R:=S^G\subseteq S$ is the invariant subring, then we study a natural homomorphism of
usually non-commutative $R$-algebras
$$
  S\ast G\,\to\,\End_R(S)
$$
and show that it is an isomorphism if the action of $G$ on $S$ is free in codimension one.
If $p>0$, then we show that $R\subseteq S$ is a pure subring if and only
if $R$ is strongly F-regular if and only if $G$ is linearly reductive if and only if 
$S\ast G$ has finite global dimension.

\subsection{Skew group scheme rings}
\label{subsec: skew group scheme rings}
Let $k$ be a field, let $S$ be a commutative $k$-algebra, and let $G\to\Spec k$
be a finite group scheme.
Assume that we have an action 
$$
   \rho \,:\,G\,\to\, \Aut_{\Spec S/\Spec k}.
$$
Then, $H^0(G,\calO_G)$ is a finite dimensional Hopf-algebra
over $k$ and we denote by $H:=H^0(G,\calO_G)^*$
the dual Hopf algebra.
The action $\rho$ corresponds to a coaction of $H^0(G,\calO_G)$
on $S$ and thus, $S$ is a right $H^*$-comodule algebra.
This is equivalent to $S$ being a left $H$-module algebra and we
let 
$$
   S*G \,:=\, S\# H \,=\, S\#\left(H^0(G,\calO_G)^*\right)
$$
be the associated smash product algebra, see, 
for example, \cite[Definition 4.1.3]{Montgomery}.

\begin{Definition}
We call $S*G$ the \emph{skew group scheme ring}.
\end{Definition}

\begin{Example}
Let $G_{\mathrm{abs}}$ be a finite group that acts on a $k$-algebra
$S$.
Let $G\to\Spec k$ be the constant group scheme associated to $G_{\mathrm{abs}}$.
Then, $H:=H^0(G,\calO_G)^*$ is isomorphic to the group algebra $k[G_{\mathrm{abs}}]$
with its usual Hopf algebra structure.
From this and \cite[Example 4.1.6]{Montgomery}, we conclude that $S*G$ as just defined
coincides with the classical skew group ring $S*G_{\mathrm{abs}}$.
\end{Example}

\subsection{Invariant rings}
With assumptions and notations from the previous paragraph, we let
$$
   R \,:=\, S^G \,\subseteq\, S
$$
be the ring of invariants with respect to the $G$-action on $S$.
In the language of Hopf algebras, these are the $H$-invariants of $S$
with $H=H^0(G,\calO_G)^*$ as defined, for example, in  \cite[Definition 1.7.1]{Montgomery}.

The multiplication $S\times S\to S$ and the $G$-action on $S$ are $R$-linear
and thus, we obtain morphisms $S\to\End_R(S)$ and $H\to\End_R(S)$.
It is easy to see that these combine to 
a natural homomorphism
$$
S\ast G \,=\, S\# H\,\to\, \End_R(S)
$$
of usually non-commutative $R$-algebras.

We now specialise to the case where $S=k[[x_1,...,x_n]]$.
Moreover, we will also assume that  the action $\rho$ is \emph{free in codimension one}, that is,
there exists a Zariski-closed subset $Z\subset\Spec S$ of codimension 
at least two, such that there is an induced action
$G\times V\to V$ where $V:=\Spec S\backslash Z$ 
and where
action is free in the scheme sense.
More precisely, if $\pi:\Spec S\to\Spec R$ denotes the quotient morphism by the $G$-action,
then we set $U:=\Spec R\backslash\pi(Z)$
and then, \emph{freeness} means that the morphism
$$
  V\times_{\Spec k} G \,\to\, V\times_U V
$$
defined by $(v,g)\mapsto(v,gv)$
is an isomorphism.
In the language of Hopf algebras, this means that if $\Spec S'$ is an affine open subset of $V$ which is stable under the $G$-action, then the $H$-action on $S$ with
$H=H^0(G,\calO_G)^*$ is \emph{Galois} in the sense of
\cite[Definition 8.1.1]{Montgomery}.
Moreover, for each $\mathfrak{p}\in U$, the $H$-action on 
$S_{\mathfrak{p}}:=S\otimes_{R} R_{\mathfrak{p}}$ is Galois. 

The following generalises a classical result of Auslander \cite[page 118]{Auslander}.
A similar statement in a slightly different setup
was already shown \cite[Proposition 2.26]{FIOS}, but since
our proof does not proceed via quotient stacks, we decided to give it 
nevertheless.

\begin{Proposition}
\label{prop: endoring}
With notations and assumptions as in Setup \ref{Setup},
$R$ is normal and the natural morphism
$$
S\ast G \,\to\, \End_R(S)
$$
is an isomorphism of  $R$-algebras.
\end{Proposition}

\begin{proof}
We let $H:=H^0(G,\calO_G)^*$ be the dual Hopf algebra
and then $R$ is the ring of invariants of $S$ with respect to the $H$-action.
Let $K$ be the field of fractions of $R$.
Since taking invariants is compatible with localisation \cite[Lemma 1.1]{Skryabin}, 
we conclude
$$
  R \,=\, R\,\cap\, (S\otimes_RK)^{H},
$$
which shows that $R$ is normal.

Therefore, $S$ is finite and reflexive as $R$-module \cite[Lemma 15.23.20]{Stacks}.
It then follows from \cite[Lemma 15.23.8]{Stacks} that $\End_R(S)$ is a reflexive $R$-module.
Since $S*G$ is a finite and free as an $S$-module, 
we conclude that $S\ast G$ 
is a reflexive $R$-module.

Since $S\ast G$ and $\End_R(S)$ both are reflexive $R$-modules,
it suffices to show that the natural homomorphism
$S\ast G\to \End_R(S)$ is an isomorphism
at each prime of height one of $R$.
When localising at primes of height one, the $H$-action is Galois because
$\rho$ is free in codimension one and then,
the statement follows from \cite[Theorem 1.7]{KT}, see also \cite[Theorem 8.3.3]{Montgomery}.
\end{proof}

The following generalises a result from the second-namend author \cite[Corollaries 3.3 and 6.18]{Yasuda} 
from finite groups to finite group schemes.

\begin{Theorem}
\label{thm: equivalences}
 We keep the notations and assumptions of Setup \ref{Setup}.
Let $\idealm\subseteq S$ be the maximal ideal and let 
 $\idealj\subseteq S\ast G$ be the Jacobson radical.
Then, the following are equivalent:
\begin{enumerate}
\item $R$ is a pure subring of $S$.
\item $R$ is strongly F-regular.
\item $G$ is linearly reductive.
\item $S\ast G$ has finite global dimension.
\item  $S\ast G$ has global dimension $n$.
\item $\idealj = \idealm  (S \ast G)$.
\item $\idealj \subseteq \idealm (S \ast G)$.
\end{enumerate}
In particular, if one of the above equivalent conditions holds, then $R$ is a Cohen-Macaulay ring. 
\end{Theorem}

\begin{proof}
$(1)\Rightarrow(2):$
This is \cite[Theorem 3.1]{HH}.

$(2)\Rightarrow(1):$
This follows from the fact that a strongly F-regular ring is a splinter, see,
for example, \cite[Proposition 1.4]{Hara}.

$(3)\Rightarrow (4):$
If $G$ is linearly reductive, then $H:=H^0(G,\calO_G)^*$ is semisimple,
see, for example, \cite[Proposition A.2]{Liedtke}.
By \cite[Corollary 4.2]{Yang}, we have
that $\mathrm{gl.dim}(S*G)=\mathrm{gl.dim}(S\#H)$ is finite.

$(4)\Rightarrow(3):$
From the previous proposition, the natural map $S*G \to \End _R (S)$ is an isomorphism.
This implies that $S$ is a faithful $S\ast G$-module.

Let $H:=H^0(G,\calO_G)^*$ be the dual Hopf algebra.
Let $0\neq t\in\int_H^\ell$ be a left integral.
By \cite[Corollary 2.3]{Yang}, there exists a $c\in S$ with $tc=1$.
If $\idealm_S=(x_1,...,x_n)$ denotes the maximal ideal of $S$, we let
$c':=\overline{c}\in S/\idealm_S=k$.
Since $tc=1$ in $S$ we also have $tc'=1$ in $k$.
Applying \cite[Corollary 2.3]{Yang} to $(S/\idealm_S)\#H=H$, we conclude
$\mathrm{gl.dim}(H)=0$.
This implies that $H$ is semi-simple
(see, for example, \cite[Theorem 4.2.2]{Weibel})
and using \cite[Proposition A.2]{Liedtke}, we conclude that
$G$ is linearly reductive.

$(3)\Rightarrow(1):$
By \cite[proof of Corollary 1.8]{Satriano}, 
we may assume that the $G$-action on $S$ is linear.
The statement then follows from \cite[Remark 6.5.3(b)]{BH}.

$(1)\Rightarrow(4):$
We have $S\ast G\cong \End_R(S)$ by Proposition \ref{prop: endoring}. 
In particular, $\End_R(S)$ is free and hence, Cohen-Macaulay as an $S$-module. 
By \cite[Proposition (1,8)]{Yoshino}, it is also Cohen-Macaulay as an $R$-module. 
By \cite[Corollary 2.11]{Yasuda} and Assumption (1), it is 
an NCCR and in particular, it has finite global dimension.

$(4)\Rightarrow(5):$
See \cite[Proposition 12.7]{Leuschke}.

$(5)\Rightarrow(4):$
Obvious.

$(6)\Rightarrow(7):$
Obvious.

$(3)\Rightarrow(6):$
We have $J(S*G)\supseteq \idealm(S \ast G)$ by \cite[(5.9)]{LamFirstCourse}. 
On the other hand, since $k*G$ is semisimple, we have
$$
J\left(S*G/(\idealm (S \ast G))\right) \,=\, J(k*G) \,=\, 0.   
$$
We also have $J(S*G) \subseteq \idealm (S \ast G)$ by \cite[15.6]{AF},
which proves equality.

$(7)\Rightarrow(4):$
We consider the pullback functor $\mathbf{F}^*$ from the module category of $S*G$ 
to the module category of $S^{1/p}*G$ that is induced by the inclusion 
$S\ast G \hookrightarrow S^{1/p}\ast G$. 
Note that the same proof as the one of \cite[Proposition 6.19]{Yasuda} 
shows that this functor is identical to the functor 
$$
  \Hom_R (S,S^{1/p}) \otimes - \quad : \quad
  \End_R(S)\text{-mod} \,\to\, \End_{R^{1/p}}(S^{1/p})\text{-mod}.
$$
By \cite[Proposition 6.17]{Yasuda}, this functor is exact, preserves projective modules,
and has zero kernel. 
By \cite[Proposition 6.13 or Corollary  6.18]{Yasuda}, the functor is also order-raising. 
Thus, $S\ast G$ has finite global dimension by \cite[Corollary 5.6]{Yasuda}.

This completes the proof of the claimed equivalences. 
The last assertion of the theorem follows, for example, from \cite[Theorem 2.6]{HH}.
\end{proof}

\begin{Remark}
  If $R$ is strongly F-regular, then it is log terminal \cite{HaraWatanabe}. 
  Thus, in the situation of the above theorem, if $G$ is linearly reductive, 
  then the quotient singularity $\Spec R^G$ is log terminal. 
  If a Noetherian local ring $R$ of equicharacteristic zero admits an NCCR, then 
  $X=\Spec R$ is log terminal, see \cite{StaffordVanDenBergh, IngallsYasuda}, as well as
   \cite{DaoIyamaTakahashiWemyss}.
   On the other hand, there are quotient singularities in positive characteristic that 
   are not log terminal, see, for example \cite{YasudaDiscrepancy}.
\end{Remark}

\begin{Remark}
  In the situation of Theorem \ref{thm: equivalences}, even if $R$ is not strongly 
  F-regular, then it can be F-pure \cite[Section 3b]{HSY} (see also \cite[Section 4]{ArtinRDP} and \cite[page 64]{ArtinWild}) 
  or F-rational \cite{HashimotoModular}. 
  There is also the case where $R$ is neither log-canonical nor Cohen-Macaulay \cite{YasudaDiscrepancy} -
  in particular, it is neither F-pure nor F-rational.
\end{Remark}

\begin{Corollary}
\label{cor: Morita}
 With notations and assumptions as in Theorem \ref{thm: equivalences}, 
 assume that the equivalent statements hold.
 Then, 
 \begin{enumerate}
 \item $\End_R(S)$ is an NCCR of $R$.
 \item For sufficiently large $e$,
 $$
    \End_R(R^{1/p^e})
 $$
 is Morita equivalent to $\End_R(S)$ and thus, it is also an NCCR of $R$.
 \end{enumerate}
\end{Corollary}

\begin{proof}
Assertion (1) was already shown in the proof of the theorem.
Assertion (2) follows from \cite[Corollary 4.2]{Yasuda}.
\end{proof}

\section{Auslander's results and dimension two}
\label{sec: Auslander}

We keep the notations and assumptions of Setup \ref{Setup} and assume moreover
that $G$ is linearly reductive.
In this section, we show that the category $\calP$ of projective $S\ast G$-modules is equivalent to
the category $\mathrm{Rep}_k(G)$ of $G$-representations, 
as well as to the category $\mathrm{add}_R(S)$ generated by direct summands of the $R$-module $S$.
If moreover $n=2$, then these are equivalent to the category $\calL$ of reflexive $R$-modules.
This generalises classical results of Auslander \cite{AuslanderRS}.

\subsection{Auslander's results}

In \cite[Section 1]{AuslanderRS}, the following proposition is shown in the case where $n=2$ 
and where $G$ is a finite group, whose order is prime to the characteristic of $k$.

\begin{Proposition}
 \label{prop: auslander1}
 We keep the notations and assumptions of Setup \ref{Setup}, let
  $\idealm\subseteq S$ be the maximal ideal, and assume that $G$ is
  linearly reductive.
 We define the two categories
 $$
\begin{array}{lcl}
\mathrm{Rep}_k(G) & : & \mbox{finite-dimensional $k$-linear $G$-representations}\\
\calP & \qquad : \qquad& \mbox{finite and projective $S\ast G$-modules},
\end{array}
$$
which are related as follows.
 \begin{enumerate}
   \item If $W\in\mathrm{Rep}_k(G)$, then
    there is a natural $S\ast G$-module structure on $S\otimes_k W$ that extends the $S$-action on $S$ and
    the $G$-action on $W$.
    Moreover, $S\otimes_k W$ is a finite and projective $S\ast G$-module, that is, lies in $\calP$.
    \item If $P$ is a finite $S\ast G$-module, then $P/\idealm P$ is a finite-dimensional
    and $k$-linear $G$-representation, that is, lies in $\mathrm{Rep}_k(G)$.
    \item If $P\in\calP$, then there exists an isomorphism
    of $S\ast G$-modules
    $$
       P \,\cong\, S\otimes_k (P/\idealm P).
    $$
   \item The functor
   $$
     \begin{array}{ccc}
        \mathrm{Rep}_k(G) &\to&  \calP \\
        W&\mapsto&S\otimes_kW
      \end{array}
    $$
    induces an equivalence of categories.
   Simple $G$-representations correspond to 
   indecomposable $S\ast G$-modules under this equivalence.
   \end{enumerate}
\end{Proposition}

\begin{proof}
We have that $S*G$ is semiperfect by \cite[(23.3)]{LamFirstCourse}.
The Jacobson radical of $S*G$ is equal to $\idealm*G$
by Theorem \ref{thm: equivalences}.
The map that sends $P$ to $P/\idealm P$ defines a bijection from isomorphism classes of indecomposable 
projective $S*G$-modules to isomorphism classes of simple $k*G$-modules by \cite[Proposition 27.10]{AF}.
Since 
$$
 \left( S\otimes _k (P/\idealm P) \right) / \idealm \left(S\otimes _k (P/\idealm P)\right) \,=\, 
 P/\idealm P,
$$
the map $W\mapsto S\otimes_k W$ is the inverse of the above bijection. 
\end{proof}

We continue with the following result, 
which in similar contexts is sometimes called \emph{Auslander's projectivisation}.

\begin{Proposition}
\label{prop: auslander3}
 We keep the assumptions and notations of Proposition \ref{prop: auslander1}
 and define the category
 $$
\begin{array}{lcl}
 \mathrm{add}_R(S) & \qquad : \qquad&  \mbox{summands of finite sums of $S$}.
\end{array}
$$
Then, the functor
$$
\begin{array}{ccc}
\calP &\to&  \mathrm{add}_R(S) \\
P&\mapsto&P^G
\end{array}
$$
induces equivalences of categories.
\end{Proposition}

\begin{proof}
See, for example, \cite[Proposition 2.4]{Leuschke}.
Note that this proposition relies on \cite[Proposition II.2.1]{ARS}, which is stated for Artinian rings only. 
However, the elementary proof there also works for Noetherian rings, which is sufficient for our purposes.
\end{proof}

Under this equivalence, the regular (resp. trivial) representation of $G$ in $\mathrm{Rep}_k(G)$
gets mapped to $S$ (resp. $R$).
Let $\rho_i:G\to\GL(V_i)$ be the set of finite-dimensional, $k$-linear, and simple representations
of $G$ up to isomorphism.
We have the well-known decomposition of the regular representation 
\begin{equation}
\label{eq: regrep}
\rho_{\mathrm{reg}} \,\cong\, \bigoplus_i\, \rho_i{}^{\oplus \dim_k\rho_i},
\end{equation}
see, for example, 
\cite[Theorem 5.9, Corollary 5.11 and Remark 5.12]{EH}.
Alternatively, one can also use the lifting result \cite[Proposition 2.9]{LMM} together with 
the decomposition \eqref{eq: regrep} of the regular representation of the finite group $G_{\mathrm{abs}}$.
Using this decomposition and Proposition \ref{prop: auslander1}, we conclude that
$$
   S \,=\, \bigoplus_i\,  \left( (S\otimes_k V_i)^G \right)^{\oplus\dim_k\rho_i}
$$
is the decomposition of $S$ into indecomposable and reflexive $R$-modules.
Applying Proposition \ref{prop: endoring} to this, we conclude the following.

\begin{Corollary}
 Under the assumptions of Proposition \ref{prop: auslander3}, there exists an isomorphism
 of $R$-algebras
$$
 S\ast G \,\cong\, \End_R\left(  \bigoplus_i \left( (S\otimes_k V_i)^G \right)^{\oplus\dim_k\rho_i} \right).
$$
\end{Corollary}

\subsection{Dimension 2}
Now, we specialise further to the case $n=2$.
Then, $R$ is normal and two-dimensional, and thus, a finite $R$-module is reflexive if and only if it
is Cohen-Macaulay, see \cite[Proposition 1.4.1]{BH}.
In particular, $S$ is a reflexive $R$-module
and thus, the objects of $\mathrm{add}_R(S)$ are reflexive $R$-modules.

\begin{Proposition}
\label{prop: auslander2}
  Under the assumptions and notations of Proposition \ref{prop: auslander2},
  assume $n=2$ and define the category
 $$
\begin{array}{lcl}
 \calL & \qquad : \qquad&  \mbox{finite and reflexive $R$-modules}.
\end{array}
$$
 Then,
\begin{enumerate}
\item up to isomorphism, the indecomposable reflexive $R$-modules are 
precisely the indecomposable $R$-summands of $S$.
\item The inclusion
$$
 \mathrm{add}_R(S) \,\subseteq\,\calL.
$$
is an equivalence of categories.
In particular, there exist only a finite number of nonisomorphic
indecomposable reflexive $R$-modules.
\end{enumerate}
\end{Proposition}

\begin{proof}
In \cite[Proposition 2.1]{AuslanderRS}, this is shown in the case where $G$ is a finite group, whose
order is prime to the characteristic of $k$.
The same arguments also work for $G$ a finite and linearly reductive group scheme
over $k$.
\end{proof}

\section{F-blowups}
\label{sec: Fblowups}

In this section, we study two-dimensional LRQ singularities and prove that they can
be resolved by sufficiently high F-blowups.
We show this by relating F-blowups to $G$-Hilbert schemes.

\subsection{F-blowups}
We start by recalling F-blowups that were introduced by the second-named author
in \cite{YasudaUF} and which are characteristic $p$ variants of higher Nash blowups.
More precisely, let $X$ be an $n$-dimensional
variety over a perfect field $k$ of characteristic $p>0$,
let $X_{\mathrm{sm}}\subseteq X$ be the smooth locus, and let 
$F^e:X_e\to X$ be the $e$.th iterated Frobenius.
For a $K$-rational point $x\in X_{\mathrm{sm}}(K)$, the fibre $(F^e)^{-1}(x)$ is a zero-dimensional
subscheme of length $p^{en}$ of $X_e\otimes _k K$ and thus, corresponds to a
$K$-rational point of the Hilbert scheme $\mathrm{Hilb}_{p^{en}}(X_e)$.

\begin{Definition}
The \emph{$e$.th F-blowup}, denoted $\mathrm{FB}_e(X)$, is the closure of 
$$
  \left\{ (F^e)^{-1}(x) \,|\, x\in X_{\mathrm{sm}} \right\}
$$
inside $\mathrm{Hilb}_{p^{en}}(X_e)$.
\end{Definition}

By \cite[Corollary 2.6]{YasudaUF}, there exists a natural morphism
$$
   \pi _e \,:\, \mathrm{FB}_e(X)\,\to\,X,
$$
which is projective, birational, and an isomorphism over $X_{\mathrm{sm}}$.
If we set $\calM_e := (F^e)_* \calO_{X_e}$, then $(\pi_e)^* \calM_e/(\text{tors})$ is locally free. 
Moreover, $\pi_e$ is the universal proper birational morphism having this property. 
In other words, the $e$.th F-blowup of $X$ is the blowup at the module 
$\calM_e$, see \cite{OnetoZatini, Villamayor}.

\subsection{G-Hilbert schemes}
Now, we assume that we are in the situation of Setup \ref{Setup}.
If $G$ is linearly reductive, then there exists a \emph{Hilbert-Chow morphism}
$$
    \mathrm{Hilb}^G(\Spec S) \,\to\, \Spec R,
$$
as shown by the first-named author in \cite[Section 4.3]{Liedtke}, which extends
the classical results of Ito and Nakamura \cite{IN}.
Existence of the $G$-Hilbert scheme for $G$ a finite and linearly reductive group scheme 
is due to Blume \cite{Blume}.

\begin{Lemma}
  The $G$-Hilbert scheme is the blowup at the $R$-module $S$.
\end{Lemma}

\begin{proof}
Let $Y\to \Spec R$ be the blowup at the $R$-module $S$. 
We will show that the birational correspondence between $Y$ and $\mathrm{Hilb}^G(\Spec S)$ 
extends to morphisms in both directions. 
If $U$ denotes the universal family over $\mathrm{Hilb}^G(\Spec S)$, then
we have inclusions
$$
U \,\subseteq\, \mathrm{Hilb}^G(\Spec S) \times _{\Spec R} \Spec S 
\,\subseteq\,  \mathrm{Hilb}^G(\Spec S) \widehat{\times} _{\Spec k} \Spec S . 
$$
Thus, $\mathrm{Hilb}^G(\Spec S)\to  \Spec R$ is a flattening of the $R$-module $S$. 
Using the universality of $Y$, we obtain a morphism $ \mathrm{Hilb}^G(\Spec S) \to Y$. 
On the other hand, the flat $Y$-scheme 
$(Y \times_{\Spec R} \Spec S)_{\mathrm{red}} \subseteq Y\widehat{\times}_{\Spec k}\Spec S$ 
induces a morphism $Y\to \mathrm{Hilb}^G(\Spec S)$. 
\end{proof}

\begin{Theorem}
\label{thm: yasuda}
 We keep the notations and assumptions of Setup \ref{Setup} and
 assume moreover that $G$ is linearly reductive.
 Then, for each $e\geq1$, we have a natural morphism
 $$
  \psi_e\,:\,\mathrm{Hilb}^G(\Spec S) \,\to\, \mathrm{FB}_e(\Spec R).
$$
Moreover, it is an isomorphism if $e$ is sufficiently large.
\end{Theorem}

\begin{proof}
If $Y_1$ and $Y_2$ are the blowups at $R$-modules $M_1$ and $M_2$ respectively, 
then the blowup at $M_1\oplus M_2$
is the unique irreducible component of  $Y_1\times_{\Spec R}Y_2$ that surjects onto $\Spec R$. 
This shows that the blowup at a module $M$ depends only on the set of 
isomorphism classes of indecomposable modules that appear as direct summands of $M$. 
It follows also that if every indecomposable summand of $M$ also appears as a direct summand of $M'$, 
then there exists a natural morphism from the blowup at $M'$ to the blowup at $M$.
This gives the existence of the natural morphism $\psi_e$, as desired.

For $e\gg 0$, the indecomposable $R$-modules that appear as summands of $R^{1/p^e}$ 
are the same as direct summands of $S$ \cite[Proposition 4.1]{Yasuda}, which shows that  $\psi_{e}$ is an isomorphism.
\end{proof}

\begin{Remark}
 The fact that $\psi_e$ is an isomorphism for sufficiently large $e$ 
 can be viewed as a commutative version of the Morita equivalence between the two 
 NCCRs $\End_R(S)=S\ast G$ and $\End_R(R^{1/p^e})$, which we 
 established in Corollary \ref{cor: Morita}.
 We refer the interested reader to \cite{TY} for a discussion
 in the case where $G$ is a finite group of order prime to $p$.
\end{Remark}

If $n=2$, then $\mathrm{Hilb}^G(\Spec S)\to\Spec R$ is the minimal resolution of singularities, which is 
due to Ishii, Ito, and Nakamura \cite{Ishii, IN, IshiiN} if $G$ is a finite group of order prime to $p$
and which is due to the first-named author \cite[Theorem 4.5]{Liedtke} if $G$ is a finite and linearly reductive group 
scheme.
Together with the previous results, we conclude the following.

\begin{Corollary}
\label{cor: minres}
Under the assumptions of Theorem \ref{thm: yasuda} 
 assume moreover $n=2$.
 If $e$ is sufficiently large, then
$$
  \mathrm{Hilb}^G(\Spec S) \,\cong\, \mathrm{FB}_e(\Spec R)
  \,\to\,\Spec R
$$
 is the minimal resolution of singularities of $\Spec R$.
\end{Corollary}

\begin{Corollary}
  \label{cor: crep dim 3}
  Under the assumptions of Theorem \ref{thm: yasuda} 
   assume moreover that $n=3$ and $R$ is Gorenstein.
   If $e$ is sufficiently large, then
  $$
    \mathrm{Hilb}^G(\Spec S) \,\cong\, \mathrm{FB}_e(\Spec R)
    \,\to\,\Spec R
  $$
   is a crepant resolution of singularities of $\Spec R$.
  \end{Corollary}

\begin{proof}
  This follows from Thereom 6.3.1 of \cite{vdBergh}, 
  in the same way as Corollary 3.6 of \cite{TY} does.
\end{proof}

We end this section by the following lemma, which gives an easy
criterion for when an LRQ singularity is Gorenstein.
In characteristic zero or if $G$ is a finite group of order prime to $p$, 
then this is due to Watanabe \cite{Watanabe}.
The extension to the linear reductive case should be known to 
the experts, but we could not find a proper reference.
See also the survey in \cite[Section 3.9.5]{Derksen}.

\begin{Proposition}
 In Setup \ref{Setup} assume moreover that $G$ is linearly reductive.
 After a change of coordinates, we may assume that the $G$-action is linear, that is,
 that $G$ is a subgroup scheme of 
 $\GL_{n,k}$ and that the $G$-action on
 $S=k[[x_1,...,x_n]]$ is compatible with the embedding of $G$ into $\GL_{n,k}$ and the usual 
 $\GL_{n,k}$-action on $S$.
The following are equivalent:
 \begin{enumerate}
  \item $G$ is a subgroup scheme of $\SL_{n,k}$.
 \item $R=S^G$ is Gorenstein.
 \end{enumerate}
\end{Proposition}

\begin{proof}
The assertion on linearisation follows from \cite[proof of Corollary 1.8]{Satriano}, see also
the discussion in \cite[Section 6.2]{LMM}.

$(1)\Rightarrow (2)$:
This is due to Hashimoto
\cite[Corollary 32.5]{LipmanHashimoto}.

$(2)\Rightarrow (1)$:
Since $G$ is linearly reductive, there exists a lift of $G$ and the linear 
$G$-action on $S=k[[x_1,...,x_n]]$ to $W(k)[[x_1,...,x_n]]$, see the discussion in \cite[Section 4.4]{Liedtke}.
From this, we obtain the canonical lift $\mathcal{X}_{\mathrm{can}}\to \Spec W(k)$ of the
LRQ singularity $X=\Spec S^G$ over the ring $W(k)$ of Witt vectors.
We let $K$ be the field of fractions of $W(k)$ and let $\overline{K}$ be an algebraic closure of $K$.

Seeking a contradiction, suppose that $G$ (which is naturally a subgroup scheme of $\GL_{n,k}$) 
is not a subgroup scheme of $\SL_{n,k}$ and that $X$ is Gorenstein.
There exists a finite field extension $L\supseteq K$, such that the generic fibre of the lift of $G$ over $L$
is a constant group scheme associated to some finite group $G_{\mathrm{abs}}$.
Since $G$ is not a subgroup scheme of $\SL_{n,k}$, we have that $G_{\mathrm{abs}}$ is
not a subgroup of $\SL_{n,\overline{K}}$ and thus,
a fortiori, not of $\SL_{n,L}$, see also the discussion in \cite[Section 4.4]{Liedtke}.
Let $\mathcal{X}_L:=\mathcal{X}\otimes_{\Spec W(k)}\Spec L$ be the generic fibre of $\mathcal{X}$
base-changed to $L$.
Since $\mathcal{X}_L$ is isomorphic to
$\Spec L[[x_1,...,x_n]]^{G_{\mathrm{abs}}}$, it is not Gorenstein by the characteristic zero 
results already mentioned above.
On the other hand, $\mathcal{X}$ is Gorenstein since $X$ is \cite[Theorem 23.4]{Matsumura},
which implies that the geometric generic fibre of $\mathcal{X}_L$ over $L$ is also Gorenstein
\cite[Theorem 18.2 and Theorem 23.6]{Matsumura}.
This is a contradiction.
\end{proof}

\section{Examples}
\label{sec: applications}

In this section, we apply the results of the previous sections to a couple of classes of singularities,
such as $\mathbb{Q}$-factorial toric singularities, F-regular surface singularities,
and canonical surface singularities.

By a \emph{singularity} $X$, we mean in this section the spectrum $X=\Spec R$ where $R$ is a local and complete
$k$-algebra with $k$ an algebraically closed field.

\subsection{Toric singularities}
Let us say that a normal singularity $X$ is \emph{toric} if arises as the
completion of a normal toric variety 
at a closed point.
The following result should be well-known to the experts, see \cite[Theorem 11.4.8]{CoxLittleSchenck} 
in characteristic zero.

\begin{Proposition}
Let $X$ be a normal $n$-dimensional singularity over an algebraically closed field $k$.
Set $S:=k[[x_1,...,x_n]]$ and let $\TT^n:=(\GG_{m,k})^n$ be the $n$-dimensional torus together 
with its usual $k$-linear action on $S$.
Then, the following are equivalent:
\begin{enumerate}
\item $X$ is toric and $\QQ$-factorial.
\item $X$ is isomorphic to $\Spec R$ with $R\cong S^G$,
where $G$ is a finite subgroup scheme of $\TT^n$, and where $G$ acts via the $\TT^n$-action on $S$ with an action that is 
free in codimension one.
\end{enumerate}
\end{Proposition}

\begin{proof}
We follow the proof of \cite[Proposition 7.3]{LMM} and generalise it to our situation:

$(1)\Rightarrow(2):$
If $X$ is toric, then it is analytically isomorphic to $\Spec k[M]$ for some affine semi-group $M$. 
Clearly, we may assume that $X$ is not smooth
and then, $\Spec k[M]$ has no torus factors.
In this situation, the Cox construction (see, for example \cite[Section 3.1]{GS}) realises $\Spec k[M]$ 
as a quotient $\Aff^n_k/G$,
where $G\cong\mathrm{Cl}(k[M])^D\cong\mathrm{Cl}(X)^D$ and where the $G$-action is linear and diagonal
(see, for example, \cite[Proposition 5.8 and Corollary 5.9]{GS}).
Since $X$ is $\QQ$-factorial, $\mathrm{Cl}(X)$ is finite and thus, $G$ is a finite group scheme.
Moreover, $G$ is a subgroup scheme of $\TT^n$.
By the linearly reductive version of the Chevalley-Todd theorem \cite{Satriano}, the $G$-action is
small, that is, free in codimension one.

$(2)\Rightarrow(1):$
We have that $X$ is toric by \cite[Theorem 5.2]{GS}.
Let $G'$ be the associated reduced subscheme of $G$, which is a finite group and let $R':=S^{G'}$. 
Then, as is well-known, $R'$ is $\QQ$-factorial. 
For $e\gg 0$, we have 
$$
 (R')^{p^e} \,\subseteq\, R \,\subseteq\, R'.
$$
Now, if $D$ is a Weil divisor of $\Spec R$ and if $D'$ is its pullback to $\Spec R'$, 
then for some $n>0$, $nD'$ is Cartier and thus, defined by some element $f\in R'$. 
Thus, $p^e n D$ is a Cartier divisor, which is defined by $f^{p^e}\in R$. 
This shows that $X$ is $\QQ$-factorial.
\end{proof}

\begin{Remarks}
 Let us make a couple of comments:
 \begin{enumerate}
  \item
 A toric variety $X=X(\Delta)$ is $\QQ$-factorial if and only if each cone $\sigma\in\Delta$ is simplicial,
 see, for example, \cite[Lemma 14-1-1]{Matsuki}. 
 In particular, normal toric varieties of dimension $n\leq2$  are $\QQ$-factorial.
 \item
 Finite subgroup schemes of $\GG_{m,k}$ are kernels of multiplication-by-$N$ for some $N\geq0$
 and thus, isomorphic to $\bmu_N$.
 Similarly, finite subgroup schemes of $\TT^n=(\GG_{m,k})^n$ are of the form 
 $\prod_{i=1}^n\bmu_{N_i}$ for some $N_i$'s with $N_i\geq0$. 
 In particular, they are diagonalisable.
 \item If $G$ is a subgroup scheme of $\TT^N$ for some $N$ and it acts on $\Spec S$
 freely in codimension one, then we may assume that the $G$-action is linear 
 because $G$ is linearly reductive.
 Since $G$ is diagonalisable, simultaneous diagonalisation implies that $G$ is a subgroup
 scheme of $\TT^n$ and that the $G$-action on $S$ is factors through the usual 
 $\TT^n$-action on $S$.
 \end{enumerate}
\end{Remarks}

\begin{Corollary}
Let $X=\Spec R$ be a normal $n$-dimensional, toric, and $\QQ$-factorial 
singularity over an algebraically closed
field of characteristic $p>0$.
\begin{enumerate}
\item If $R=S^G$ with $S=k[[x_1,...,x_n]]$ and $G$ as in 
Theorem \ref{thm: F-regular}.(2), then $\End_R(S)$ is an NCCR of $R$.
\item For $e$ sufficiently large, $\End_R(R^{1/p^e})$ is an NCCR of $R$.
\item If $n=2$ and if $e$ is sufficiently large, then $\mathrm{FB}_e(X)$ is the minimal 
resolution of singularities of $X$.
\item If $n=3$, if $e$ is sufficiently large, and if $R$ is Gorenstein, then
$\mathrm{FB}_e(X)$ is a crepant resolution of singularities of $X$.
\end{enumerate}
\end{Corollary}

\begin{proof}
Assertions (1) and (2) follow from Corollary \ref{cor: Morita}.
Assertions (3) and (4) follow from Corollary \ref{cor: minres} and
Corollary \ref{cor: crep dim 3}, respectively.
\end{proof}

\begin{Remark}
Faber, Muller, and Smith \cite{Faber} proved that if $\Spec R$ is a normal toric singularity in characteristic $p>0$
and if $e$ is sufficiently large, then $\End_R(R^{1/p^e})$ is an NCR.
Recall that this means that the latter ring has finite global dimension, but that it is not necessarily Cohen-Macaulay.
A similar result in characteristic zero was established by \v{S}penko and van den Bergh in \cite{Spenko}. 
\end{Remark}

\subsection{F-regular surface singularities}
\label{subsec: F-regular surfaces}
Let us recall the following result from \cite[Section 11]{LMM}:

\begin{Theorem}
\label{thm: F-regular}
Let $X$ be a normal two-dimensional singularity over an algebraically closed
field of characteristic $p>0$.
Then, the following are equivalent
\begin{enumerate}
\item $X$ is F-regular (resp. Gorenstein and F-regular).
\item $X$ is the quotient singularity by a finite and linearly reductive subgroup scheme
$G$ of $\GL_{2,k}$ (resp. $\SL_{2,k}$).
\end{enumerate}
Moreover, if $p\geq7$, then this is equivalent to
\begin{enumerate}
\setcounter{enumi}{2}
\item $X$ is a log terminal (resp. canonical) singularity.\qed
\end{enumerate}
\end{Theorem}

By the results of the previous sections, we thus obtain the following.

\begin{Corollary}
Let $X=\Spec R$ be a normal two-dimensional and F-regular singularity over an algebraically closed
field of characteristic $p>0$.
\begin{enumerate}
\item If $R=S^G$ with $S=k[[x_1,x_2]]$ and $G$ as in 
Theorem \ref{thm: F-regular}.(2), then $\End_R(S)$ is an NCCR of $R$.
\item If $e$ is sufficiently large, then $\End_R(R^{1/p^e})$ is an NCCR of $R$.
\item If $e$ is sufficiently large, then $\mathrm{FB}_e(X)$ is the minimal resolution of singularities of $X$.\qed
\end{enumerate}
\end{Corollary}

\begin{Remark}
  Assertion (3) is a theorem of Hara \cite{Hara},
  which we recover here in the context of LRQ singularities and
   $G$-Hilbert schemes.
\end{Remark}

\subsection{Canonical surface singularities}
If $X$ is a canonical surface singularity over an algebraically closed field $k$
of characteristic $p\geq0$, then it is a rational double point.
If $p>0$, then these have been classified by Artin \cite{ArtinRDP}
and they are all of the form $X=\Spec R$ with
$$
  R \,=\, k[[x_1,x_2,x_3]]/(f)
$$
for a suitable polynomial $f=f(x_1,x_2,x_3)$.
If $p\geq7$, then the results recalled in Section \ref{subsec: F-regular surfaces}
show that all canonical surface singularities are F-regular and LRQ
singularities and thus, the results about NCCRs and F-blowups of the
previous sections apply.
However, if $p<7$, then not all canonical surface singularities are F-regular.
The following result is more or less well-known.

\begin{Proposition}\label{prop:can-surf-sing}
Every canonical surface singularity over an algebraically closed
field admits an NCCR.
\end{Proposition}

\begin{proof}
Let $\Spec R$ be a canonical surface singularity over an algebraically closed field.
By \cite{ArtinVerdier}, there are only finitely many  indecomposable maximal Cohen-Macaulay modules over $R$ 
up to isomorphism. 
Let $M$ be the direct sum of all of them.  
By \cite[Theorem 6]{LeuschkeEndomorphismRings}, $\End_R(M)$ is an NCCR.
\end{proof}

\begin{Remarks}
If $0<p\leq5$, then not all canonical surface singularities are LRQ singularities.
\begin{enumerate}
\item If $p=5$, then the singularities $E_8^0$ and $E_8^1$ (notation as in \cite{ArtinRDP}) 
are quotient singularities 
 $$
  R\,=\,S^G \,\subseteq\, S\,=\,k[[x_1,x_2]]
$$
with $G$ isomorphic to $\balpha_5$ and $\C_5=\ZZ/5\ZZ$, respectively.
This is in contrast to the group that is usually assigned to $E_8$-singularities if $p=0$ or $p\geq7$, 
namely the binary icosahedral group, which is a non-abelian group of order 120.
We have $S*G\cong\End_R(S)$ by Proposition \ref{prop: endoring}, but this ring does not have 
finite global dimension by Theorem \ref{thm: equivalences}.
Thus, although these singularities admit NCCRs, they are not given by $\End_R(S)$.
\item If $p=3$, then the singularity $E_8^0$ is not a quotient singularity
by \cite[Theorem 1.12]{LMM2}.
\item If $X$ is a canonical surface singularity, then $\mathrm{FB}_e(X)$
has only rational singularities and it is dominated by the minimal resolution of $X$ 
by \cite[Proposition 3.2]{HSY}.
However, it is not necessarily true that $\mathrm{FB}_e(X)$ is a resolution of singularities even if 
$e$ is sufficiently large, see \cite[Theorem 1.1]{HSY} for
a couple of examples, which include the $E_8^0$-singularity in characteristic $p=5$.
This has to do with the fact that for every $e\geq1$, there exists an indecomposable maximal Cohen-Macaulay module of $R$
that is not a summand of $R^{1/p^e}$.
Thus, an NCCR $\End_R(M)$ of such a singularity, which exists by Proposition \ref{prop:can-surf-sing}, 
is not of the form $\End_R(R^{1/p^e})$.
\end{enumerate}
\end{Remarks}

\end{document}